\documentclass[12pt,leqno]{article}
\usepackage{amsfonts,amsthm,amsmath}

\usepackage[mathscr]{eucal}
\usepackage{amsmath}
\usepackage{amsfonts}
\usepackage{amssymb}
\usepackage{footnpag}

\theoremstyle{plain}
\newtheorem{thm}{Theorem}[section]
\newtheorem{prop}{Proposition}[section]

\theoremstyle{definition}
\newtheorem{df}{Definition}[section]
\newtheorem{rem}{Remark}[section]

\newcommand{\Z}{\mathbb{Z}}

\newcommand{\C}{\mathbb{C}}

\newcommand{\ZZ}{\mathbb{Z}}
\newcommand{\RR}{\mathbb{R}}
\newcommand{\CC}{\mathbb{C}}

\newcommand{\1}{{\bf 1}}

\def\h{{\mathfrak{h}}}
\def\hh{\hat{\mathfrak{h}}}

\newcommand{\Aut}{{\rm Aut}}
\newcommand{\w}{\omega}
\newcommand{\ot}{\otimes}
\newcommand{\op}{\oplus}
\newcommand{\eqa}{\begin{eqnarray}}
\newcommand{\eeqa}{\end{eqnarray}}
\newcommand{\eqn}{\begin{eqnarray*}}
\newcommand{\eeqn}{\end{eqnarray*}}

\newcommand{\tr}{{\rm tr}}

\newcommand{\ra}{\rangle}
\newcommand{\la}{\langle}

\begin{document}

\title{Conformal designs and 
D.H. Lehmer's conjecture\footnote{This work was supported by JSPS KAKENHI.}}

\author{
Tsuyoshi Miezaki\thanks{Systems Science and Information Studies, 
Faculty of Education, Art and Science, 
Yamagata University, 
1-4-12 Kojirakawa, Yamagata 990-8560, Japan. email: miezaki@e.yamagata-u.ac.jp
}
}
\date{}
\maketitle

\begin{abstract}
In 1947, Lehmer conjectured that the Ramanujan $\tau$-function $\tau (m)$ is non-vanishing for all positive integers $m$, where $\tau (m)$ are the Fourier coefficients of the cusp form $\Delta$ of weight $12$. 
It is known 
that Lehmer's conjecture can be reformulated in terms of 
spherical $t$-design, by the 
result of Venkov. 
In this paper, we show that $\tau (m)=0$ is equivalent to 
the fact that the homogeneous space of the moonshine vertex operator algebra $(V^{\natural})_{m+1}$ is a conformal $12$-design. 
Therefore, Lehmer's conjecture is now reformulated in terms of 
conformal $t$-designs. 
\end{abstract}
\vspace{-0.15in}

\quad

\section{Introduction}
In \cite{H1}, H\"ohn defined the concept of conformal designs, 
which is an analogue of the concept of combinatorial designs and 
spherical designs. 
First, we review some information that will be needed later in the present paper. 
See \cite{Bo}, \cite{FHL}, and \cite{FLM} for the 
definitions and elementary information on vertex operator algebras 
and their modules. 

A vertex operator algebra (VOA) $V$ over the field $\CC$ of 
complex numbers is a complex vector space 
equipped with a linear map $Y : V \rightarrow {\rm End}(V)[[z, z^{-1}]]$ 
and two non-zero vectors $\1$ and $\omega$ in $V$ 
satisfying certain axioms (cf. \cite{{FHL},{FLM}}). 
We denote a VOA $V$ by $(V,Y,\1,\w)$. 
For $v \in V$, we write 
\[
Y(v,z) =\sum_{n\in\ZZ}v(n)z^{-n-1}. 
\]
In particular, for $\w \in V$, we write 
\[
Y(\w,z) =\sum_{n\in\ZZ}L(n)z^{-n-2}, 
\]
and $V$ is graded by $L(0)$-eigenvalues: 
%
$V=\oplus_{n\in \Z} V_{n}$. 
For $V_n$, $n$ is called the degree. 
In the present paper, 
we assume that 
$V_{n}=0$ for $n<0$, 
and $V_{0}=\C \bold{1}$. For $v \in V_{n}$, 
the operator $v(n-1)$ is homogeneous of degree $0$. 
We set $o(v) = v(n-1)$. 
We also assume that the VOAs $V$ 
are isomorphic to a direct sum of 
highest weight modules for the Virasoro algebra, 
i.e., 
\begin{align}\label{eqn:decom}
V=\bigoplus_{n\geq 0}V(n), 
\end{align}
where each $V(n)$ is a sum of the 
highest weight $V_\w$ modules of 
highest weight $k$ and $V(0)=V_\w$.


In particular, the decomposition (\ref{eqn:decom}) yields
the natural projection map 
\begin{align*}
\pi : V \rightarrow V_{\omega} 
\end{align*}
with the kernel $\oplus_{n>0}V(n)$. 
Next, we give the definition of a conformal $t$-design, which 
is based on Matsuo's paper \cite{Matsuo}. 
\begin{df}[cf.~\cite{H1}]\label{df:con}
Let $V$ be a VOA of central charge $c$, 
and let $X$ be a degree $h$ subspace of a module of $V$. 
For a positive integer $t$, $X$ is referred to as a conformal $t$-design 
if for all $v \in V_{n}$, where $0 \leq n \leq t$, we have
\begin{align*}
{\rm tr}\vert _{X} o(v) = {\rm tr}\vert _{X} o(\pi (v)). 
\end{align*}
\end{df}

Then, it is easy to prove the following theorem: 
\begin{thm}[cf.~{\cite[Theorem 2.3]{H1}}]\label{thm:2.3}
Let $X$ be the homogeneous subspace of a module of 
a VOA $V$.
Then, the following conditions are equivalent$:$
\begin{enumerate}
\item 
[{\rm (i)}] 

$X$ is a conformal $t$-design.

\item 
[{\rm (ii)}] 

For all homogeneous $v \in \ker \pi 
=\bigoplus_{n>0}V(n)$
of degree $n \leq t$, one has 
$\tr|_{X}o(v) = 0$. 
\end{enumerate}
\end{thm}
$V_m$ can be considered to have large symmetry if a homogeneous space of VOA $V_m$
is a conformal $t$-design for higher $t$ \cite{Matsuo}. 

We next give examples of conformal designs. 
Let $V^{\natural}$ be the moonshine VOA {\rm \cite{FLM}}. 
Then, it is well known that $(V^{\natural})_m$ is 
a conformal $11$-design for all $m$ \cite{H1}. 
Therefore, it is an interesting 
problem to prove or disprove the existence of a conformal $12$-design 
which is a homogeneous space of $V^{\natural}$. 

The main result in this paper is as follow: 
\begin{thm}\label{thm:main1}
Let the notation be the same as above. 
Let $\tau(m)$ be Ramanujan's $\tau$-function$:$ 
\begin{eqnarray}\label{eqn:delta}
\Delta(z)=\eta(z)^{24}=(q^{1/24}\prod_{m\geq 1}(1-q^{m}))^{24}
=\sum_{m\geq 1}\tau (m)q^{m}.
\end{eqnarray}
Then, the following are equivalent{\rm :} \begin{enumerate}
\item 
[\rm{(i)}]

$\tau (m)=0$. 
\item 
[\rm{(ii)}]

$(V^{\natural})_{m+1}$ is a conformal $12$-design.
\end{enumerate}
\end{thm}


Lehmer conjectured that $\tau (m) \neq 0$ \cite{Lehmer}. 
Thus, Theorem \ref{thm:main1} is 
a reformulation of Lehmer's conjecture. 

A homogeneous space of VOA $V_m$ 
has strength $t$ 
if $V_m$ is a conformal $t$-design 
but is not a conformal $(t+1)$-design. 
We have not yet been able determined 
the strength of $(V^{\natural})_m$ for general $m$, 
and so Lehmer's conjecture is still open. 
This demonstrates the difficulty of determining 
the strength of $V_m$.
However, we will give examples for which the strength $t$ can be determined: 
\begin{thm}\label{thm:main2}
The homogeneous spaces in a $d$-free boson VOA have 
strength $3$. 
\end{thm}

In Sections $2$ and $3$, respectively, Theorems 1.2 and 1.3 will be proved.

%

\begin{rem}
In the theory of spherical designs, analogues of Theorem \ref{thm:main1} have been obtained \cite{{Venkov},{Pache}}. 
Namely, $\tau (m)=0$ is equivalent to 
the shell of $E_8$-lattice $(E_8)_{2m}$ being a spherical $8$-design \cite{BM1} 
(see, e.g., \cite{BM1} for the undefined terms in this remark). 
\end{rem}

\section{The case of $V^{\natural}$}
\subsection{Graded trace}
In this section, we review the concept of a graded trace. 
Recall that $V$ is a VOA with standard $L(0)$-grading 
\[
V=\bigoplus_{n\geq 0}V_{n}. 
\]
Then, for $v\in V_k$, we define 
the graded trace $Z_{V}(v,z)$ as follows: 
\[
Z_{V}(v, z) = \tr\vert_{V} o(v)q^{L(0)-c/24} = 
q^{-c/24}\sum_{n=0}^{\infty}(\tr\vert_{V_{n}}o(v))q^{n}, 
\]
where $c$ is the central charge of $V$. 
If $v=\1$, then
\[
Z_{V}(\1,z) = \tr\vert_{V} q^{L(0)-c/24} 
= q^{-c/24}\sum_{n=0}^{\infty}(\dim V_{n})q^{n}. 
\]



For a VOA $V=(V,Y,\1,\w)$, 
Zhu defined the new VOA $(V, Y[\ ], \1,\w - c/24)$, 
where $c$ is the central charge of $V$ \cite{Zhu}. 
Let $\widetilde{\w}=\w - c/24$ and 
\[
Y[\widetilde{\w}, z] =
\sum_{n\in \ZZ}L[n]z^{-n-2}. 
\]
Then, we have $V=\bigoplus_{n=0}^{\infty}V_{[n]}$ and 
\begin{align}\label{eqn:Zhu}
\bigoplus_{n\leq N}V_{n}
=\bigoplus_{n\leq N}V_{[n]}. 
\end{align}

\subsection{Proof of Theorem \ref{thm:main1}}
Note that the moonshine VOA is an
extremal self-dual VOA in the sense of \cite{H1}, 
namely, $(V^{\natural})_{1}=0$ 
and 
for $v\in (V^{\natural})_{[12]}$, 
we have 
$Z_{V^{\natural}}(v,z)$ is a cusp form of weight $12$ 
with respect to $SL_2(\ZZ)$ (cf.~\cite[page 299, line 11 up]{Zhu}, 
\cite{DLM}). 
It is well known that $\Delta (z)$ is the unique cusp form of weight $12$ 
with respect to $SL_{2}(\ZZ)$. 
Therefore, we have 
$Z_{V^{\natural}}(v,z)=c(v)\Delta (z)$, 
where $c(v)$ is a constant depend on $v$ 
(cf.~\cite[page 299, line 11 up]{Zhu}, 
\cite{DLM}). 
Assume that $\tau (m)=0$. 
Then, for any $v\in (V^{\natural})_{[12]}$, 
we have $\tr\vert_{(V^{\natural})_{m+1}}o(v)= 0$. 
Therefore, 
based on (\ref{eqn:Zhu}), 
$(V^{\natural})_{m+1}$ is a conformal $12$-design.

Now, we assume the contrary, that $\tau (m)\neq 0$. 
Since 
$(V^{\natural})_{2}$ is not a conformal $12$-design 
(cf~\cite[page 2333]{H1},~\cite[Theorem $3$]{DM}), 
there exists $v\in (V^{\natural})_{[12]}$ 
of degree $12$ 
such that 
$Z_{V^{\natural}}(v,z)=c(v)\Delta(z)
=c(v)\sum_{m=1}^{\infty}\tau(m)q^{m}$, 
where $c(v)\neq 0$ (cf.~\cite{Zhu}). 
Hence, we have 
$\tr\vert_{(V^{\natural})_{m+1}}o(v)=c(v)\times \tau(m)\neq 0$, 
which implies that 
$(V^{\natural})_{m+1}$ is not a conformal $12$-design, by (\ref{eqn:Zhu}). 
This completes the proof of Theorem \ref{thm:main1}. 

\section{The case of $M(1)$}
\subsection{Free boson vertex operator algebras}
In this section, we review the definition of 
the $d$-free boson VOA $M(1)$. 
For details of the construction, see \cite{FLM}. 
Let $\h$ be a $d$-dimensional vector space with 
a nondegenerate symmetric bilinear form $(,)$, 
and let $\hh$ be the corresponding affinization, 
viewing $\h$ as an abelian Lie algebra 
$\hh=\h\ot \CC[t, t^{-1}]\op \CC K$ with commutator relations 
\begin{align*}
[h\ot t^{m}, h^{\prime}\ot t^{n}]&=m(h, h^{\prime})\delta_{m+n, 0}K, 
(h, h^{\prime}\in \h, m, n \in\ZZ), \\
[K, h\ot \CC[t, t^{-1}]]&=0. 
\end{align*}
Consider the induced module 
\[
M(1)=\mathcal{U}(\hh)\ot_{\h\ot\CC[t]\op\CC K}\CC, 
\]
where $\h\ot\CC[t]$ acts trivially on $\CC$, and $K$ acts as $1$. 
We denote by $h(n)$ the action of $h\ot t^{n}$ on $M(1)$. 
The space $M(1)$ is linearly isomorphic to 
the symmetric algebra $S(\h \ot t^{-1}\CC[t^{-1}])$. 
Thus, setting $\1 = 1 \ot 1$, any element in $M(1)$ 
is a linear combination of elements of type
\[
v = a_1(-n_1 ) \cdots a_k(-n_k)\1,\ 
(a_1, \ldots , a_k \in \h,\ n_1, \ldots , n_k \in \ZZ_{+}). 
\]

Now let $\{h_i\}_{i=1}^{d}$ be an orthonormal basis of $\h$, 
and set $\omega = 1/2\sum_{i=1}^{d}h_{i}(-1)^{2}\1$. 
Then, $(M(1), Y, \1,\omega)$ is a VOA with a vacuum $\1$ and 
Virasoro element $\w$. 
In particular,
\[
M(1)=\bigoplus_{n\geq 0}M(1)_{n}, 
\]
where $M(1)_{n}=\la a_1(-n_1) \cdots a_k(-n_k)\1\mid 
a_1, \ldots , a_k \in \h,\ n_1, \ldots , n_k \in \ZZ_{+},\ 
\sum n_i=n\ra$. 
We identify $M(1)_{1}$ with $\h$ in the obvious way. 

\subsection{Proof of Theorem \ref{thm:main2}}
First, note that the automorphism group of 
a $d$-free boson VOA is the orthogonal group $O(d,\CC)$ \cite{DM2}. 

\begin{prop}\label{prop:ex}
For $k>0$, $(M(1))_{k}$ is a conformal $3$-design. 
\end{prop}
\begin{proof}
Let $G=O(d,\RR)< O(d,\CC)$ be a subgroup of the automorphism group of $M(1)$. 
Let $\theta$ be an element in $G$ of order $2$ 
that is a lift of $-1\in \Aut(\h)$, namely, 
for $a_1(-n_1)\cdots a_k(-n_k)\1\in M(1)$,
\[
\theta: a_1(-n_1)\cdots a_k(-n_k)\1
\mapsto(-1)^{k}a_1(-n_1)\cdots a_k(-n_k)\1. 
\] 
Then, 
\begin{align*}
\left\{
\begin{array}{l}
(M(1)^{\la\theta\ra})_1=\emptyset \\ 
(M(1)^{\la\theta\ra})_2=S^2(\frak{h}(-1))\\
(M(1)^{\la\theta\ra})_3=\h(-2)\ot \h(-1). 
\end{array}
\right.
\end{align*}
Therefore, based on~{\cite[Theorem 2.5]{H1}}, it is sufficient to show that 
$(M(1)^{G})_2=\CC\w=\CC(\sum_{i=1}^{d}h_i(-1)\ot h_i(-1))$ and 
$(M(1)^{G})_3=\CC L(-1)\w=\CC(\sum_{i=1}^{d}h_i(-2)\ot h_i(-1))$. 
We consider the $G$-action on $\CC[x_1,\ldots, x_d]$. 
Then, the invariants $\CC[x_1,\ldots, x_d]^G$ are the space $\CC[x_1^2+\cdots+x_d^2]$, \cite{Er}. 
Substituting $\{x_i\}_{i=1}^d$ for $\{h_i\}_{i=1}^d$, 
we obtain the desired results. 
Hence, 
$(M(1))_{k}$ is a conformal $3$-design. 
\end{proof}
\begin{prop}\label{prop:nonex}
For $k>0$, $(M(1))_{k}$ is not a conformal $4$-design. 
\end{prop}

\begin{proof}
Let $\{h_{i}\}_{i=1}^{d}$ be the orthonormal basis of $\h$, and 
let $$v_{4}=h_1(-1)^4\1-2h_1(-3)h_1(-1)\1+\frac{3}{2}h_1(-2)^2\1.$$ 
Then, $v_{4}\in (M(1))_4$ is the highest weight vector 
because $L(1)v_{4}=L(2)v_{4}=0$ (cf.~\cite[page~423]{DMN}). 
Then, it is sufficient to show that 
\[
\tr\vert_{(M(1))_{k}}o(v_{4})\neq 0. 
\]
We set $a(m)$ as follows: 
\[
Z_{M(1)}(v_{4},z)=q^{-d/24}\sum_{m\geq 0}
({\rm tr}\vert _{(M(1))_{m}}o(v_{4}))q^{m}
=q^{-d/24}\sum_{m\geq 0}a(m)q^{m}. 
\]
We show that $a(m)\neq 0$. 
We will divide the problem into two cases: $d=1$ and $d\geq 2$. 
Let $d=1$. 
For $1\leq k\leq 3$, by calculation, we have
\[
\tr\vert_{(M(1))_{k}}o(v_4)=
\left\{
\begin{array}{ll}
-6& {\rm if\ } k=1\\
-42&{\rm if\ } k=2\\
-120&{\rm if\ } k=3, 
\end{array}
\right. 
\]
that is, 
\begin{align}
Z_{M(1)}(v_{4},z)=q^{-1/24}(-6q-42q^2-120q^3+\cdots).\label{eqn:1}
\end{align}
On the other hand, based on {\cite[Theorem 1]{DMN}}, 
\begin{align}
Z_{M(1)}(v_{4},z)=\frac{f_1(v_{4},z)}{\eta(z)},\label{eqn:2}
\end{align}
where $f_1(v_{4},z)\in \CC[E_{2}, E_{4}, E_{6}]$, and 
since $v_{4}\in (M(1))_{4}$ and based on {\cite[Theorem 1]{DMN}}, 
$f_1(v_{4},z)$ can be written in terms of $E_{2}$, $E_{2}^2$, and $E_{4}$. 
Therefore, we can determine $f_1(v_{4},z)$ 
by (\ref{eqn:1}) and (\ref{eqn:2}) as follows: 
\begin{align*}
f_1(v_{4},z)=\frac{E_{2}(z)^2-E_{4}(z)}{48}. 
\end{align*}
Here, using the equation (see \cite{Hahn}) 
\begin{align*}
\frac{E_{4}(z)-E_{2}(z)^{2}}{288}=\sum_{m>0}m\sigma_{1}(m)q^{m}, \\
\end{align*}
we obtain 
\begin{align*}
f_1(v_{4},z)=\frac{E_{2}(z)^2-E_{4}(z)}{48}=(-6)\sum_{m>0}m\sigma_{1}(m)q^{m}. 
\end{align*}
Therefore, the coefficients of $f_1(v_4,z)$ are negative integers. 
Since the coefficients of $1/\eta(z)$ are positive integers, 
$a(m)\neq 0$ for all $m>0$ for the case $d=1$. 

Let $d\geq 2$ and 
\begin{align}
Z_{M(1)}(v_{4},z)=\frac{f_d(v_{4},z)}{\eta(z)^d}, \label{eqn:3}
\end{align}
where $f_d(v_{4},z)\in \CC[E_{2}, E_{4}, E_{6}]$ and 
$f_d(v_{4},z)$ is also written in terms of $E_{2}$, $E_{2}^2$, and $E_{4}$. 
Using~{\cite[Corollary 2.2.2]{DMN}}, we have that 
$f_1(v_{4},z)$ and $f_d(v_{4},z)$ 
are same function because $v_{4}$ is written in terms of 
$h_1(-1)$, $h_1(-2)$, and $h_1(-3)$, and 
the basis $\{h_i\}_{i=2}^{d}$ does not influence 
{\cite[Corollary 2.2.2]{DMN}}, namely, 
\begin{align*}
f_d(v_{4},z)=\frac{E_{2}(z)^2-E_{4}(z)}{48}. 
\end{align*}
Using the same argument for the case $d=1$, 
we have $a(m)\neq 0$. 
\end{proof}
Propositions \ref{prop:ex} and \ref{prop:nonex} are summarized as Theorem \ref{thm:main2}. 

\bigskip
\noindent
{\bf Acknowledgment.}
The author would like to thank Gerald H\"ohn, Hiroki Shimakura, and Kenichiro Tanabe for their helpful discussions. 

\end{document}